\documentclass[10pt,a4paper]{article}
\usepackage[utf8]{inputenc}
\usepackage{amsmath}
\usepackage{amsfonts}
\usepackage{amssymb}
\usepackage{makeidx}
\usepackage{graphicx}
\usepackage{color}
\usepackage{xspace}
\usepackage{multicol}

\definecolor{azul}{rgb}{0.0, 0.0, 1.0}
\definecolor{rojo}{rgb}{1.0, 0.03, 0.0}

\newtheorem{theo}{Theorem}

\newtheorem{theorem}{Theorem}

\newtheorem{definition}[theorem]{Definition}

\newtheorem{lemma}[theo]{Lemma}

\newtheorem{proposition}[theo]{Proposition}

\newenvironment{proof}[1][Proof]{\noindent\textbf{#1.} }{\ \rule{0.5em}{0.5em}}

\begin{document}

\title{Integral operators with rough kernels in
	variable Lebesgue spaces}
\author{Marta Urciuolo - Lucas Vallejos}
\maketitle

\begin{abstract}
	In this paper we study integral operators with kernels 
	\begin{equation*}
	K(x,y)= k_1( x- A_1y)...k_m( x-A_my),
	\end{equation*}
	$k_i(x)=\frac{\Omega_i(x)}{|x|^{n/q_i}}$ where $\Omega_i: \mathbb{R}^n\to 
	\mathbb{R}$ are homogeneous functions of degree zero, satisfying a size and
	a Dini condition, $A_{i}$ are certain invertible matrices, and \ $\frac
	n{q_1}+\dots\frac n{q_m}=n-\alpha,$ $0\leq \alpha <n.$\newline
	We obtain the boundedness of this operator from $L^{p(\cdot)}$ into $%
	L^{q(\cdot)}$ for $\frac{1}{q(\cdot)}=\frac{1}{p(\cdot)}-\frac{\alpha }{n},$
	for certain exponent functions $p$ satisfying weaker conditions than the classical log-Hölder conditions.
	
	\footnote{%
		Partially supported by CONICET and SECYTUNC}  \footnote{%
		}  \footnote{%
		}
\end{abstract}

\section{Introduction}

Given a measurable set $\Omega\subseteq \mathbb{R}^n$ we denote with $\mathcal{P}(\Omega)$ the family of measurable function $p(\cdot):\Omega \rightarrow [1, \infty]$. Given $p(\cdot)\in \mathcal{P}(\Omega)$ let $L^{p(\cdot)}(\Omega)$ be the Banach space of measurable
functions $f$ on $\Omega$ such that for some $\lambda >0,$ 
\begin{equation*}
\rho_{p(\cdot),\Omega}(f/\lambda) <\infty ,
\end{equation*}
where 
\begin{equation*}
\rho_{p(\cdot),\Omega}(f)=\int_{\Omega\backslash \Omega_{\infty}} \left\vert f(x)\right\vert^{p(x)}dx + \lVert f \rVert_{L^{\infty}(\Omega_{\infty})},
\end{equation*}
$\Omega_{\infty}=\left\lbrace x \in \Omega : p(x)=\infty \right\rbrace$, with norm
\begin{equation*}
\left\Vert f\right\Vert _{L^{p(\cdot)}(\Omega)}=\inf \left\{ \lambda >0: \rho_{p(\cdot),\Omega}(f/\lambda)\leq 1\right\}.
\end{equation*}%
We will denote $\lVert f \rVert_{p(\cdot)}$ instead of $\lVert f \rVert_{L^{p(\cdot)}(\Omega)}$ if the role of the set $\Omega$ is clear enough.
These spaces are known as \textit{variable exponent spaces} and are a
generalization of the classical Lebesgue spaces $L^{p}(\mathbb{R}^{n}).$
They have been widely studied lately. See for example \cite{CCF}, \cite%
{C-F-N} and \cite{DHHR}. The first step was to determine sufficient
conditions on $p(\cdot)$ for the boundedness on $L^{p(\cdot)}$ of the Hardy
Littlewood maximal operator 
\begin{equation*}
\mathcal{M}f(x)=\sup\limits_{B}\frac{1}{\left\vert B\right\vert }%
\int_{B}\left\vert f(y)\right\vert dy,
\end{equation*}%
where the supremun is taken over all balls $B$ containing $x$. 
Analogously for $0<\alpha <n$ we recall the definition of the fractional maximal operator
\begin{equation*}
M_{\alpha }f(x)=\sup_{Q\ni x}\frac{1}{\left\vert B\right\vert^{1-\frac{\alpha}{n}}}\int_{B}\left\vert f(y)\right\vert dy,
\end{equation*}
where also the supremun is taken over all balls $B$ containing $x$.\newline
We also define $M_{\alpha ,s}$, for all $1\leq s < \infty$, 
\begin{equation}
\label{Malfas}
M_{\alpha ,s}f=(M_{\alpha .s}\left\vert f\right\vert ^{s})^{1/s}.
\end{equation}
Let $p_{-}=ess\inf $\textit{\ }$p(x)$ and let $p_{+}=ess\sup $ $p(x)$.
\begin{definition}
	Given $\Omega \subseteq \mathbb{R}^{n}$, and \ function $r(\cdot ):\Omega
	\rightarrow \mathbb{R}$, we say that $r(\cdot )$ is locally log-Hölder
	continuous, and denote this by $r(\cdot )\in LH_{0}(\Omega )$, if there
	exist a constant $C_{0}$ such that for all $x,y\in \Omega $, $\left\vert
	x-y\right\vert <\frac{1}{2}$,
	\begin{equation*}
	\left\vert r(x)-r(y)\right\vert \leq \frac{C_{0}}{-\log (\left\vert
		x-y\right\vert )}.	
	\label{log-local}
	\end{equation*}
	We say that $r(\cdot )$ is log-Hölder continuous at infinity, and denote
	this by $r(\cdot )\in LH_{\infty }(\Omega )$, if there exist constants $
	C_{\infty }$ and $r^{\infty }$ such that for all $x\in \Omega $,
	
	\begin{equation*}
	\left\vert r(x)-r^{\infty }\right\vert \leq \frac{C_{\infty }}{\log
		(e+\left\vert x\right\vert )}\text{.}
	\label{log-infinito}
	\end{equation*}
	
\end{definition}

In \cite{C-F-N} , D. Cruz Uribe, A. Fiorenza and C. J. Neugebauer proved the
following result. If  $p(\cdot) \in \mathcal{P}(\mathbb{R}^n)$, $1<p_{-}\leq p_{+}<\infty$ and  $p(\cdot )\in LH_{0}(\mathbb{R}%
^{n})\cap LH_{\infty }(\mathbb{R}^{n})$, then the Hardy Littlewood maximal
operator is bounded on $L^{p(\cdot)}(\mathbb{R}^{n}).$ In \cite{C-F} (chapter 4)
the authors show that the boundedness of the maximal operator can be
obtained under weaker conditions on the exponent $p(\cdot ).$ They define
the $N_{\infty }-condition$ as follows, 

\begin{definition}
	Given $\Omega\subseteq \mathbb{R}^n$ and $
	p(\cdot)\in \mathcal{P(\mathbb{R}}^{n}\mathcal{)}$, we say that $%
	p(\cdot) \in N_{\infty}(\Omega)$ if there exist constants $%
	\Lambda_{\infty}$ and $p_{\infty}$ such that  
	\begin{equation*}
	\int_{\Omega_+}\exp \left( -\Lambda_{\infty} \left| \dfrac{1}{p(x)} - \dfrac{%
		1}{p_{\infty}} \right| ^{-1} \right) dx < \infty.
	\end{equation*}
	where $\Omega_+ = \left\lbrace x \in \Omega : \left| \dfrac{1}{p(x)} - 
	\dfrac{1}{p_{\infty}} \right|>0 \right\rbrace$.
\end{definition}

Also, in \cite{C-F}, the authors define the $K_{0}-condition$ as follows,

\begin{definition}
	Given $%
	p(\cdot )\in \mathcal{P(\mathbb{R}}^{n}\mathcal{)}$, then $p(\cdot )\in
	K_{0}(\mathbb{R}^{n})$ if there exists a constant $C$ such that, for
	every cube $Q$,
	
	\begin{equation*}
	\left\Vert \chi _{Q}\right\Vert _{L^{p(\cdot )}(\mathbb{R}^{n})}\left\Vert
	\chi _{Q}\right\Vert _{L^{p^{\prime }(\cdot )}(\mathbb{R}^{n})}\leq
	C\left\vert Q\right\vert 
	\end{equation*}
\end{definition}

They prove the following result. If $p(\cdot)\in \mathcal{P}(\mathbb{R}^n) $, $1<p_{-}\leq p_{+}<\infty$ and $p(\cdot )\in K_{0}(%
\mathbb{R}^{n})\cap N_{\infty }(\mathbb{R}^{n})$, then the Hardy Littlewood
maximal operator is bounded on $L^{p(.)}(\mathbb{R}^{n}).$ They also show
that $LH_{0}(\mathbb{R}^{n})\cap LH_{\infty }(\mathbb{R}^{n})\subset K_{0}(%
\mathbb{R}^{n})\cap N_{\infty }(\mathbb{R}^{n})$ and they give an example
that shows that the the inclusion is strict.\newline
Let $0\leq \alpha <n$, $m\in \mathbb{N}$. For $1\leq i\leq m,$ let $%
1<q_{i}<\infty $ such that $\frac{n}{q_{1}}+\dots +\frac{n}{q_{m}}=n-\alpha .
$ For $\alpha =0$ we take $m>1.$We denote by $\Sigma =\Sigma _{n-1}$ the
unit sphere in $\mathbb{R}^{n}$. Let $\Omega _{i}\in L^{1}(\Sigma )$. If $%
x\neq 0$, we write $x^{\prime }=x/|x|$. We extend this function to $\mathbb{R%
}^{n}\setminus \{0\}$ as $\Omega _{i}(x)=\Omega _{i}(x^{\prime })$.\newline 

Let 
\begin{equation}
\label{ki}
k_{i}(x)=\frac{\Omega _{i}(x)}{|x|^{n/q_{i}}},
\end{equation}%
and let
\begin{equation}
\label{T}
T_{\alpha }f(x)=\int_{\mathbb{R}^{n}}K(x,y)f(y)dy,
\end{equation}
with $K(x,y)=k_{1}(x-A_{1}y)...k_{m}(x-A_{m}y)$, where $A_{i}$, are certain
invertible matrices and $f\in L_{loc}^{\infty }(\mathbb{R}^{n})$. In \cite{Ri-U} the authors consider the operator $T_\alpha$ defined in (\ref{T}) where, for $1\leq i\leq m$, $k_i$ is given by (\ref{ki}). For $1\leq p\leq \infty$ and $\Omega_i \in L^1(\Sigma)$, they define the $L^{p}$- modulus of continuity as 
\begin{equation*}
\varpi_{_i, p}(t)=\sup_{|y|\le t}
\|\Omega_i(\cdot+y)-\Omega_i(\cdot)\|_{p,\Sigma}.
\end{equation*}
They make the following hypothesis about the functions $\Omega_i$, $1\leq
i\leq m$,

\ 

($H_1$)\ \ There exists $p_i>q_i$ such that $\Omega _i\in L^{p_i}(\Sigma)$,

\ 

($H_{2}$)\ \ $\displaystyle\int_{0}^{1}\varpi _{i,p_{i}}(t)\frac{dt}{t}%
<\infty .$\newline

They obtain the boundedness of this kind of operators in weighted Lebesgue spaces. We recall that a weight $\omega $ is a locally integrable and non negative function. The Muckenhoupt class $\mathcal{A}_{p}$, $1<p<\infty $, is defined as the class of weights $\omega $ such that

\begin{equation*}
\sup_{Q} \left[ \left( \frac{1}{\lvert Q \rvert} \displaystyle \int_Q \omega
\right) \left( \frac{1}{\lvert Q \rvert} \displaystyle \int_Q \omega^{-\frac{
		1}{p-1}} \right)^{p-1} \right] < \infty,
\end{equation*}
where $Q$ is a cube in $\mathbb{R}^n$.

For $p=1$, $\mathcal{A}_1$ is the class of weights $\omega$ satisfying that
there exists $c>0$ such that

\begin{equation*}
\mathcal{M} \omega (x) \leq c \omega (x) \ a.e. \ x \in \mathbb{R}^n.
\end{equation*}

We denote $\left[ \omega \right]_{\mathcal{A}_1}$ the infimum of the
constant $c$ such that $\omega$ satisfies the above inequation.

In this paper we study the boundedness of $T_{\alpha }$ on variable Lebesgue
spaces. The exponent functions will satisfy certain regularity conditions and
also certain relations with the different matrices $A_{i}$ involved in the
kernel $K$. We will ask the hypothesis $p(A_i x)\leq p(x) \ a.e.x \in \mathbb{R}^n$. In \cite{U-V-2} we proved that this condition is in fact necessary in some particular cases.
We will first prove the boundeness of the fractional maximal operator in variable Lebesgue spaces, with standard extrapolation techniques. Then we obtain the $L^{p(\cdot)}(\mathbb{R}^n)-L^{q(\cdot)}(\mathbb{R}^n)$ boundedess of $T_{\alpha}$, $\frac{1}{q(\cdot)}-\frac{1}{p(\cdot)}=\frac{\alpha}{n}$. To obtain this result we use
the boundedness of the "sharp maximal function". We recall that given a function $f\in L_{loc}^{1}(\mathbb{R}^{n})$ we define the sharp maximal function 
\begin{equation*}
M^{\#}f(x)=\sup\limits_{B}\frac{1}{\left\vert B\right\vert }%
\int\limits_{B}\left\vert f(y)-\frac{1}{\left\vert B\right\vert }%
\int\limits_{B}\left\vert f\right\vert \right\vert dy,
\end{equation*}%
where $B$ is a ball containing $x.$ In \cite{C-F} the authors prove that
given $p(\cdot)$ with $1\leq p_{-}\leq p_{+}<\infty ,$ if the maximal operator
is bounded on $L^{p^{\prime }(\cdot)}\left( \mathbb{R}^{n}\right) $ then there exists $c>0$ such that 
\begin{equation*}
\sup_{t>0}\left\Vert t\chi _{\left\{ x:\left\vert f(x)\right\vert >t\right\}
}\right\Vert _{p(\cdot)}\leq c\sup_{t>0}\left\Vert t\chi _{\left\{ x:\left\vert
	M^{\#}f(x)\right\vert >t\right\} }\right\Vert _{p(\cdot)}
\end{equation*}%
and 
\begin{equation*}
\left\Vert f\right\Vert _{p(\cdot)}\leq c\left\Vert M^{\#}f\right\Vert _{p(\cdot)}.
\end{equation*}

\section{Main results}

In this paragraph we use the sharp maximal function to obtain the boundedness of the operator defined by (\ref{T}).\newline
In \cite{M-W}, B. Muckenhoupt y R.L. Wheeden define $\mathcal{A}(p,q),$ $1<p<\infty$ and $1<q<\infty$, as the class of weights $\omega$ such that
\begin{equation*}
\sup_{Q} \displaystyle\left[ \left( \dfrac{1}{\lvert Q \rvert} \displaystyle%
\int_{Q}\omega(x)^q dx \right)^{\frac{1}{q}}\left( \dfrac{1}{\lvert Q \rvert}
\displaystyle\int_{Q} \omega(x)^{-p^{\prime }}dx \right)^{\frac{1}{p^{\prime
}}}\displaystyle\right] \ < \ \infty.
\end{equation*}
When $p=1$, $\omega \in \mathcal{A}(1,q)$ if only if
\begin{equation*}
\sup_{Q} \displaystyle\left[ \lVert \omega^{-1} \chi_Q \rVert_{\infty}\left( 
\dfrac{1}{\lvert Q \rvert} \displaystyle\int_Q \omega(x)^{q}dx \right)^{%
	\frac{1}{q}}\displaystyle\right] \ < \ \infty.
\end{equation*}
They prove that $M_{\alpha}:L^p(\omega^p)\longrightarrow L^q(\omega^q)$ for weights $\omega \in \mathcal{A}(p,q)$, $\frac{1}{p}-\frac{1}{q}=\frac{\alpha}{n}$. They also obtain the corresponding weak type inequality. With classical extrapolation techniques we get the following result,
\begin{lemma}
	\label{M-F}
	Let $0<\alpha <n$ and let $p(\cdot )\in \mathcal{P(\mathbb{R}}^{n}\mathcal{)%
	}$, such that $1\leq p_{-}\leq p_{+}<\frac{n}{\alpha }$. Let $q(\cdot) \in \mathcal{P}(\mathbb{R}^n)$ be defined by $\frac{1}{p(x)}-\frac{1}{q(x)}=\frac{\alpha }{n}$ and we suppose that the maximal operator is bounded on $L^{(q(\cdot)/q_0)'}(\mathbb{R}^{n})$, with $q_0$ such that $\frac{1}{p_{-}}-\frac{1}{q_0}=\frac{\alpha}{n}$.\newline
 	If $p_{-}>1$ then there exist $C>0$ such that
	\begin{equation*}
	\left\Vert M_{\alpha }f\right\Vert _{q(\cdot )}\leq C\left\Vert f\right\Vert
	_{p(\cdot )}. 
	\end{equation*}
	If $p_{-}=1$ then, for all $\lambda >0$,
	\begin{equation*}
	\left\Vert \lambda \chi _{\{x\in \mathbb{R}^{n}:\left\vert M_{\alpha
		}f(x)\right\vert >\lambda \}}\right\Vert _{q(\cdot )}\leq C\left\Vert
	f\right\Vert _{p(\cdot )}.
	\end{equation*}
\end{lemma}
\begin{proof}
	Let $q_{0}:\frac{1}{p_{-}}-\frac{1}{q_{0}}=\frac{\alpha }{n}$ and we suppose that $p_{-}>1$. Let $\tilde{q}(\cdot )=\frac{q(\cdot )}{q_{0}}$, we take a bounded function $f$ with compact support. 
	\begin{equation*}
	\left\Vert M_{\alpha }f\right\Vert _{q(\cdot )}^{q_{0}} =\left\Vert
	(M_{\alpha }f)^{q_{0}}\right\Vert _{\tilde{q}(\cdot )}\leq C\sup_{\left\Vert
		h\right\Vert _{\tilde{q}^{\prime }(\cdot )}=1}\int \left( M_{\alpha
	}f\right) ^{q_{0}}(x)h(x)dx,
	\end{equation*}
	We define an iteration algorithm on $L^{\tilde{q}^{\prime }(\cdot)}(\mathbb{R}^{n})$ by 
	\begin{equation*}
	Rh(x)=\sum_{k=0}^{\infty }\frac{M^{k}h(x)}{2^{k}\left\Vert M\right\Vert _{%
			\tilde{q}^{\prime }(\cdot )}^{k}}.
	\end{equation*}
	As in \cite{C-F} it follows that
	
	($i$) $\left\vert h(x)\right\vert \leq Rh(x)$ for all $x\in \mathbb{R}^{n}$.
	
	($ii$) $R$ is bounded on $L^{\tilde{q}^{\prime }(\cdot )}(\mathbb{R}^{n})$
	and $\left\Vert Rh\right\Vert _{\tilde{q}^{\prime }(\cdot )}\leq 2\left\Vert
	h\right\Vert _{\tilde{q}^{\prime }(\cdot )}$
	
	($iii$) $Rh\in \mathcal{A}_{1}$ and $\left[ Rh\right] _{\mathcal{A}_{1}}\leq 2\left\Vert
	M\right\Vert _{\tilde{q}^{\prime }(\cdot )}$\newline
	So
	\begin{equation*}
	\leq C\sup_{\left\Vert h\right\Vert _{\tilde{q}^{\prime }(\cdot )}=1}\int
	\left( M_{\alpha }f\right) ^{q_{0}}(x)\left( Rh(x)^{\frac{1}{q_{0}}}\right)
	^{q_{0}}dx,
	\end{equation*}
	since $Rh \in \mathcal{A}_{1}$ then $Rh^{\frac{1}{q_0}} \in \mathcal{A}_{1} \subset \mathcal{A}(p_{-},q_0)$, and so, since $M_{\alpha }$ is bounded from $L^{p_{-}}(w^{p_{-}})$ in $L^{q_{0}}(w^{q_{0}})$ for weights $w\in \mathcal{A}(p_{-},q_{0})$,
	\begin{equation*}
	\leq C\sup_{\left\Vert h\right\Vert _{\tilde{q}^{\prime }(\cdot )}=1}\left(
	\int \left( f\right) ^{p_{-}}(x)\left( Rh(x)^{\frac{p_{-}}{q_{0}}}\right)
	dx\right) ^{\frac{q_{0}}{p_{-}}},
	\end{equation*}
	
	now we apply the H\"{o}lder's inequality with $\tilde{p}(\cdot )=\frac{%
		p(\cdot )}{p_{-}}$,
	
	\begin{equation*}
	\leq C\left\Vert f^{p_{-}}\right\Vert _{\tilde{p}(\cdot )}^{\frac{q_{0}}{%
			p_{-}}}\sup_{\left\Vert h\right\Vert _{\tilde{q}^{\prime }(\cdot
			)}=1}\left\Vert Rh^{\frac{p_{-}}{q_{0}}}\right\Vert _{\tilde{p}^{\prime
		}(\cdot )}^{\frac{q_{0}}{p_{-}}}
	\end{equation*}
	
	\begin{equation*}
	\leq C\left\Vert f\right\Vert _{p(\cdot )}^{q_{0}}.
	\end{equation*}
	We now show that $\left\Vert M_{\alpha }f\right\Vert _{q(\cdot )}<\infty $.
	By proposition 2.12 in [2] it is enough to check that $\rho _{q(\cdot
		)}(M_{\alpha }f)<\infty $.
	\begin{equation*}
	\left\vert M_{\alpha }f(x)\right\vert ^{q(x)}\leq \left\vert M_{\alpha
	}f(x)\right\vert ^{q_{+}}\chi _{\{x\in \mathbb{R}^{n}:M_{\alpha
		}f(x)>1\}}+\left\vert M_{\alpha }f(x)\right\vert ^{q_{-}}\chi _{\{x\in 
		\mathbb{R}^{n}:M_{\alpha }f(x)\leq 1\}}
	\end{equation*}
	Now, since $f$ is bounded with compact support, $M_{\alpha }f\in L^{s}(\mathbb{R}^{n})$ for all $\frac{n}{n-\alpha }<s<\infty $.\newline\newline
	The rest of the proof follows as in the proof of Theorem 5.46 in \cite{C-F}.
\end{proof}
\begin{lemma}
 	\label{desigualdad-Ninfinito}
 	Given $\Omega \subseteq \mathbb{R}^{n}$. If $p(\cdot )\in N_{\infty }(\Omega
 	)$ and $p_{\infty}=\infty$ then $1\in L^{p(\cdot )}(\Omega )$.
 \end{lemma}
 \begin{proof}
 	For $\lambda >1$ sufficiently large, by the $N_{\infty }-$ condition for $p$, and $\Omega _{+}=\Omega \smallsetminus \Omega _{\infty }$,
 	\begin{equation*}
 	\int_{\Omega \smallsetminus \Omega _{\infty }}\lambda
 	^{-p(x)}dx=\int_{\Omega \smallsetminus \Omega _{\infty }}e^{-p(x)\ln
 		(\lambda )}dx
 	\end{equation*}
 	\begin{equation*}
 		=\int_{\Omega \smallsetminus \Omega _{\infty }}e^{-\Lambda
 		_{\infty }p(x)\frac{\ln (\lambda )}{\Lambda _{\infty }}}dx\leq \int_{\Omega
 		\smallsetminus \Omega _{\infty }}e^{-\Lambda _{\infty }p(x)}dx <\infty.
 	\end{equation*}
 \end{proof}
 \begin{proposition}
 	\label{prop-Ninfty}
  	Let $A$ be an invertible matrix $n \times n$.\newline
 	$a)$ If $p(\cdot )\in N_{\infty }(\mathbb{R}^n)$, $1\leq p_{-}\leq p_{+}<\infty $ and 
 	$p(Ax)\leq p(x) \ a.e.x\in \mathbb{R}^n$, then there exists $c>0$ such that
 	\begin{equation*}
 	\left\Vert f\circ A^{-1}\right\Vert _{p(\cdot )}\leq c\left\Vert
 	f\right\Vert _{p(\cdot )},
 	\end{equation*}
 	for all $f\in L^{p(\cdot )}(\mathbb{R}^n)$.\newline
 	$b)$ If $p(Ax)=p(x) \ a.e.x\in\mathbb{R}^n$ then there exist $c>0$ such that
 	\begin{equation*}
 	\lVert f\circ A \rVert_{p(\cdot)} \leq c \lVert f \rVert_{p(\cdot)},
 	\end{equation*} 
 	for all $f\in L^{p(\cdot )}(\mathbb{R}^n)$.
 \end{proposition}
 \begin{proof}
 	$a)$ We assume that $f$ is bounded with compact support and $\left\Vert
 	f\right\Vert _{p(\cdot )}\leq 1$. We will prove that,
 	
 	\begin{equation*}
 	\left\Vert f\circ A^{-1}\right\Vert _{p(\cdot )}\leq c.
 	\end{equation*}
 	We descompose $f=f_{1}+f_{2}$ where $f_{1}=f\chi _{\left\{ x:\left\vert
 		f(x)\right\vert >1\right\} }$ and $f_{2}=f\chi _{\left\{ x:\left\vert
 		f(x)\right\vert \leq 1\right\} }$,
 	\begin{equation*}
 	\left\Vert f\circ A^{-1}\right\Vert _{p(\cdot )}\leq \left\Vert f_{1}\circ
 	A^{-1}\right\Vert _{p(\cdot )}+\left\Vert f_{2}\circ A^{-1}\right\Vert
 	_{p(\cdot )}.
 	\end{equation*}
 	We define $E=\left\{ x:p(x)\geq p_{\infty }\right\} $ and $F=\left\{
 	x:p(x)<p_{\infty }\right\} $.
 	We estimate $\left\Vert f_{2}\circ A^{-1}\right\Vert _{p(\cdot )}$,
 	\begin{equation*}
 	\left\Vert f_{2}\circ A^{-1}\right\Vert _{L^{p(\cdot )}}\leq \left\Vert
 	f_{2}\circ A^{-1}\right\Vert _{L^{p(\cdot )}(E)}+\left\Vert f_{2}\circ
 	A^{-1}\right\Vert _{L^{p(\cdot )}(F)}.
 	\end{equation*}
 	Since $f_{2}$ is bounded and with compact support, $f_{2}\in L^{p_{\infty }}(%
 	\mathbb{R}^{n})$ and so
 	
 	\begin{equation*}
 	f_{2}\circ A^{-1}\in L^{p_{\infty }}(\mathbb{R}%
 	^{n}) 
 	\end{equation*}
 	Lemma 3.28 in \cite{C-F}, with $g=f_{2}\circ A^{-1}$, $t(\cdot )=p_{\infty }$, $u(\cdot)=p(\cdot )$ implies that, if $\left\Vert f_{2}\circ A^{-1}\right\Vert _{L^{p_{\infty }}(E)}<1$ then
 	\begin{equation*}
 	\left\Vert f_{2}\circ A^{-1}\right\Vert _{L^{p(\cdot )}(E)}\leq \left\Vert
 	f_{2}\circ A^{-1}\right\Vert _{L^{p_{\infty }}(E)}+1<2.
 	\end{equation*}
 	and if $\left\Vert f_{2}\circ A^{-1}\right\Vert _{L^{p_{\infty }}(E)}\geq 1$
 	then
 	\begin{equation*}
 	\left\Vert f_{2}\circ A^{-1}\right\Vert _{L^{p(\cdot )}(E)}\leq 2\left\Vert
 	f_{2}\circ A^{-1}\right\Vert _{L^{p_{\infty }}(E)}\leq 2\det (A)\left\Vert
 	f_{2}\right\Vert _{L^{p_{\infty }}(\mathbb{R}^{n})}
 	\end{equation*}
 	\begin{equation*}
 	\leq 2\det (A)\left[
 	\left\Vert f_{2}\right\Vert _{L^{p_{\infty }}(E)}+\left\Vert
 	f_{2}\right\Vert _{L^{p_{\infty }}(F)}\right]. 
 	\end{equation*}
 	To estimate $\left\Vert f_{2}\right\Vert _{L^{p_{\infty }}(E)}$,
 	by the definition of $E$, we define the defect exponent $r(\cdot )\in
 	\mathcal{P}(E)$ by
 	\begin{equation*}
 	\frac{1}{p_{\infty }}=\frac{1}{p(x)}+\frac{1}{r(x)}. 
 	\end{equation*}
 	By the generalizated H\"{o}lder's inequality, Corollary 2.28 in \cite{C-F}, 
 	\begin{equation*}
 	\left\Vert f_{2}\right\Vert _{L^{p_{\infty }}(E)}\leq K\left\Vert
 	1\right\Vert _{L^{r(\cdot )}(E)}\left\Vert f_{2}\right\Vert _{L^{p(\cdot
 			)}(E)}\leq K\left\Vert 1\right\Vert _{L^{r(\cdot )}(E)}<\infty.
 	\end{equation*}
 	The last inequality follows since $r(\cdot)\in N_{\infty }$, $r_{\infty}=\infty$ and so Lemma \ref{desigualdad-Ninfinito} implies $1\in L^{r(\cdot )}(E)$.\newline\newline
 	To estimate $\left\Vert f_{2}\right\Vert _{L^{p_{\infty }}(F)}$,
 	we apply Lemma 3.28 in \cite{C-F}, with $g=f_{2}\in L^{p(\cdot )}(F)$, $t(\cdot
 	)=p(\cdot )$, $u(\cdot )=p_{\infty }$. Since $\left\Vert f_{2}\right\Vert _{L^{p(\cdot )}(F)}\leq 1$,
 	\begin{equation*}
 	\left\Vert f_{2}\right\Vert _{L^{p_{\infty }}(F)}\leq \left\Vert
 	f_{2}\right\Vert _{L^{p(\cdot )}(\mathbb{R}^{n})}+1\leq 2. 
 	\end{equation*}
 	Combining the above estimates we get
 	\begin{equation*}
 	\left\Vert f_{2}\circ A^{-1}\right\Vert _{L^{p(\cdot )}(E)}\leq
 	C(K\left\Vert 1\right\Vert _{L^{r(\cdot )}(E)}+2)<\infty.
 	\end{equation*}
 	Now, in a similar way, we estimate $\left\Vert f_{2}\circ A^{-1}\right\Vert _{L^{p(\cdot)}(F)}$. We define the defect exponent $s(\cdot )\in \mathcal{P}(F)$ by
 	\begin{equation*}
 	\frac{1}{p(x)}=\frac{1}{p_{\infty }}+\frac{1}{s(x)}.
 	\end{equation*}
 	By the generalized H\"{o}lder's inequality, Corollary 2.28 in \cite{C-F},
 	\begin{equation*}
 	\lVert f_{2}\circ A\rVert_{L^{p(\cdot )}(F)}^{-1}\leq K\left\Vert 1\right\Vert
 	_{L^{s(\cdot )}(F)}\left\Vert f_{2}\circ A^{-1}\right\Vert _{L^{p_{\infty
 		}}(F)}.
 	\end{equation*}
 	Since $s(\cdot )\in N_{\infty }$ and $s_{\infty}=\infty$, by Lemma \ref{desigualdad-Ninfinito} we have that $1\in L^{s(\cdot )}(F)$. Further, we can now argue as we did above to get
 	\begin{equation*}
 	\left\Vert f_{2}\circ A^{-1}\right\Vert _{L^{p_{\infty }}(F)}\leq \left\Vert
 	f_{2}\circ A^{-1}\right\Vert _{L^{p_{\infty }}(\mathbb{R}^{n})}\leq \det
 	(A)\left\Vert f_{2}\right\Vert _{L^{p_{\infty }}(\mathbb{R}^{n})}
 	\end{equation*}
 	\begin{equation*}
 	\leq \det
 	(A)\left[ \left\Vert f_{2}\right\Vert _{L^{p_{\infty }}(E)}+\left\Vert
 	f_{2}\right\Vert _{L^{p_{\infty }}(F)}\right] <\infty. 
 	\end{equation*}
 	We now estimate $\left\Vert f_{1}\circ A^{-1}\right\Vert _{p(\cdot )}$.
 	Since $p_{+}<\infty $ it's enough to prove that there exits $c>0$ such that
 	$\rho _{p(\cdot )}(f_{1}\circ A^{-1})\leq c$.
 	Since $p(Ax)\leq p(x)$ $a.e.x\in \mathbb{R}^n$ and again from Corollary 2.22 in \cite{C-F},
 	\begin{equation*}
 	\int f_{1}(A^{-1}x)^{p(x)}dx=\det (A)\int f_{1}(x)^{p(Ax)}dx\leq \det
 	(A)\rho _{p(\cdot )}(f_{1})\leq c\left\Vert f_{1}\right\Vert _{p(\cdot )}<c.
 	\end{equation*}
 	For general $f\in L^{p(.)}\left( \mathbb{R}^{n}\right) ,$ we apply Theorem
 	2.59 in \cite{C-F}. For $k\in \mathbb{N},$ we define $f^{k}(x)=\left\vert
 	f\right\vert \chi _{\left\{ x:\left\vert x\right\vert \leq k,\left\vert
 		f(x)\right\vert \leq k\right\} },$ $f^{k}(x)$ increases to $\lvert f(x) \rvert$ pointwise
 	almost everyhere and so $\left\Vert f^{k}\right\Vert _{p(.)}\rightarrow
 	\left\Vert f\right\Vert _{p(.)}$ and also $\left\Vert f^{k}\circ
 	A^{-1}\right\Vert _{p(.)}\rightarrow \left\Vert f\circ A^{-1}\right\Vert
 	_{p(.)}.$ Since each $f^{k}$ is a bounded function with compact support and $%
 	\left\Vert \frac{f_{k}}{\left\Vert f_{k}\right\Vert _{p(.)}}\right\Vert \leq
 	1$, we have proved that there exists a constant $c>0$ such that $\left\Vert 
 	\frac{f_{k}\circ A^{-1}}{\left\Vert f_{k}\right\Vert _{p(.)}}\right\Vert
 	_{p(.)}\leq c$, so 
 	\begin{equation*}
 	\left\Vert \frac{f\circ A^{-1}}{\left\Vert f\right\Vert _{p(.)}}\right\Vert
 	_{p(.)}=\lim\limits_{k\rightarrow \infty }\left\Vert \frac{f_{k}\circ A^{-1}%
 	}{\left\Vert f\right\Vert _{p(.)}}\right\Vert _{p(.)}\leq
 	\lim\limits_{k\rightarrow \infty }\left\Vert \frac{f_{k}\circ A^{-1}}{%
 		\left\Vert f_{k}\right\Vert _{p(.)}}\right\Vert _{p(.)}\leq c,
 	\end{equation*}
 	and then 
 	\begin{equation*}
 	\left\Vert f\circ A^{-1}\right\Vert _{p(.)}\leq c\left\Vert f\right\Vert
 	_{p(.)}.
 	\end{equation*}
	$b)$ Let $f\in L^{p(\cdot)}(\mathbb{R}^n)$. We have that%
\begin{equation*}
\left\Vert f\circ A\right\Vert _{p(\cdot )}=\inf \left\{ \lambda >0:\int_{%
	\mathbb{R}^{n}}\left( \frac{f(Ax)}{\lambda }\right) ^{p(x)}dx\leq 1\right\}. 
\end{equation*}
By a change of variable and using the hypothesis on the exponent, 
\begin{equation}
\label{Det}
\int_{\mathbb{R}^{n}}\left( \frac{f(Ax)}{\lambda }\right)
^{p(x)}dx=\left\vert \det(A^{-1})\right\vert \int_{\mathbb{R}^{n}}\left( \frac{f(y)}{\lambda }\right)^{p(y)}dy,
\end{equation}
Let $D=\left\vert \det (A^{-1})\right\vert $, then we have two cases:
\bigskip If $D\leq 1,$%
\begin{equation*}
\left\Vert f\circ A\right\Vert _{p(\cdot )}\leq \left\Vert f\right\Vert
_{p(\cdot )}.
\end{equation*}
If $D>1$, then from (\ref{Det}) it is follows that%
\begin{equation*}
D\int_{\mathbb{R}^{n}}\left( \frac{f(y)}{\lambda }\right) ^{p(y)}dy=\int_{%
	\mathbb{R}^{n}}\left( \frac{f(y)}{\lambda C^{\frac{1}{p(y)}}}\right)
^{p(y)}dy,
\end{equation*}
where $C=\frac{1}{D}$. So,
\begin{equation*}
\leq \int_{\mathbb{R}^{n}}\left( \frac{f(y)}{\lambda C^{\frac{1}{p_{-}}}}%
\right) ^{p(y)}dy.
\end{equation*}
That is,
\begin{equation*}
\int_{\mathbb{R}^{n}}\left( \frac{f(Ax)}{\lambda }\right) ^{p(x)}dx\leq
\int_{\mathbb{R}^{n}}\left( \frac{f(x)}{\lambda C^{\frac{1}{p_{-}}}}\right)
^{p(x)}dx.
\end{equation*}
From this last inequality it follows that%
\begin{equation*}
\left\Vert f\circ A\right\Vert _{p(\cdot )}\leq D^{\frac{1}{p_{-}}%
}\left\Vert f\right\Vert _{p(\cdot )}.
\end{equation*}
\end{proof} 
\begin{theo}
	Let $0\leq \alpha <n$ and let $T_{\alpha }$ be the integral operator given
	by (\ref T). Let $m\in \mathbb{N}$ (or $m\in \mathbb{N\setminus }\left\{
	1\right\} $ for $\alpha =0$), let $A_{1},...,A_{m}$ be invertible matrices
	such that $A_{i}-A_{j}$ is invertible for $i\neq j$, $1\leq i,j\leq m$ and
	the functions $\Omega _{i}$ satisfy the hypothesis (H1) and (H2). Let $s\geq 1
	$ be defined by $\frac{1}{p_{1}}+...+\frac{1}{p_{m}}+\frac{1}{s}=1$, let $
	p(\cdot )\in \mathcal{P(\mathbb{R}}^{n}\mathcal{)}$ be such that $1\leq
	s\leq p_{-}\leq p_{+}<\frac{n}{\alpha }$ and such that $p(A_{i}x)\leq p(x)$
	a.e.$x\in \mathbb{R}^{n}$ and let $q(\cdot )\in \mathcal{P(\mathbb{R}}^{n}%
	\mathcal{)}$ be defined by $\frac{1}{p(\cdot )}-\frac{1}{q(\cdot )}=\frac{%
		\alpha }{n}$. If $\frac{q(\cdot )}{s}\in N_{\infty }(\mathbb{R}^{n})\cap
	K_{0}(\mathbb{R}^{n})$ then,
	
	$a)$ there exist $C>0$ such that
	
	\begin{equation*}
	\left\Vert \lambda \chi _{\left\{ x:T_{\alpha }f(x)>\lambda \right\}
	}\right\Vert _{q(\cdot )}\leq C\left\Vert f\right\Vert _{p(\cdot )}
	\end{equation*}
	
	for all $\lambda>0$, $f\in L_{c}^{\infty }(\mathbb{R}^{n})$.
	
	$b)$ If $p_{-}>s$ then $T_{\alpha }$ extends to a bounded operator from $%
	L^{p(\cdot )}(\mathbb{R}^{n})$ into $L^{q(\cdot )}(\mathbb{R}^{n})$.
\end{theo}
\begin{proof}
	$a)$ In \cite{Ri-U} the authors prove that, for $f\in L_{c}^{\infty}(\mathbb{R}^n)$,
	\begin{equation}
	\label{MSharp}
	M^{\#}(T_{\alpha }f)(x)\leq c\sum_{i=1}^{m}M_{\alpha ,s}f(A_{i}^{-1}x),
	\end{equation}
	$a.e.x\in \mathbb{R}^n$, where $M_{\alpha,s}$ is defined in \ref{Malfas}. Since $\frac{q(\cdot)}{s} \in  N_{\infty}(\mathbb{R}^n)\cap K_0 (\mathbb{R}^n)$ then $q'(\cdot) \in N_{\infty}(\mathbb{R}^n)\cap K_0 (\mathbb{R}^n)$. Indeed is easy to check that if $p(\cdot) \in N_{\infty}(\mathbb{R}^n)$ then $\alpha p(\cdot) \in N_{\infty} (\mathbb{R}^n)$ for all $\alpha\geq 1$. So $s\frac{q(\cdot )}{s}=q(\cdot )\in
	N_{\infty }(\mathbb{R}^n).$ Remark 4.6 in \cite{C-F} implies $\ q^{\prime
	}(\cdot )\in N_{\infty }(\mathbb{R}^n)$. Since $\frac{q(\cdot )}{s}\in N_{\infty }(\mathbb{R}^n)\cap K_{0}(\mathbb{R}^n)$ then by Theorem 4.52 in \cite{C-F} the maximal operator it is bounded on $L^{\frac{q(\cdot )}{s}}(\mathbb{R}^n)$. So, by Theorem 4.37 in \cite{C-F}, it is bounded on  $L^{q(\cdot)}(\mathbb{R}^n)$. Also by Corollary 4.64 in \cite{C-F} it is bounded on $L^{q^{\prime }(\cdot )}(\mathbb{R}^n)$. By Corollary 4.50 in \cite{C-F}, $q^{\prime }(\cdot )\in K_{0}(\mathbb{R}^n)$. And so $q^{\prime }(\cdot )\in N_{\infty }(\mathbb{R}^n )\cap K_{0}(\mathbb{R}^n)$.\newline\newline
	Let $\lambda >0$ and $f \in L^{\infty}_c(\mathbb{R}^n)$. Since $q'(\cdot)\in N_{\infty}(\mathbb{R}^n)\cap K_0 (\mathbb{R}^n)$, again Theorem 5.52 in \cite{C-F} implies that the maximal operator is bounded on $L^{q'(\cdot)}(\mathbb{R}^n)$, so from Theorem 5.54 in \cite{C-F} and (\ref{MSharp}), 
	\begin{equation*}
	\lVert \lambda \chi _{\{x:T_{\alpha }f(x)>\lambda \}} \rVert
	_{q(\cdot )} \leq C\sup_{\lambda >0}\left\Vert \lambda \chi
	_{\{x:M^{\#}(T_{\alpha }f)(x)>\lambda \}}\right\Vert _{q(\cdot )}
	\end{equation*}
	\begin{equation*}
	\leq C\sup_{\lambda >0}\left\Vert \lambda \chi
	_{\{x:\sum_{i=1}^{m}M_{\alpha ,s}f(A_{i}^{-1}x)>\frac{\lambda }{c}
		\}}\right\Vert _{q(\cdot )} \leq C\sup_{\lambda >0}\left\Vert \lambda \sum_{i=1}^{m}\chi
	_{\{x:M_{\alpha ,s}f(A_{i}^{-1}x)>\frac{\lambda }{cm}\}}\right\Vert
_{q(\cdot )} 
	\end{equation*}

	\begin{equation*}
	\leq C\sup_{\lambda >0}\sum_{i=1}^{m}\left\Vert \lambda \chi
	_{\{x:M_{\alpha ,s}f(A_{i}^{-1}x)>\frac{\lambda }{cm}\}}\right\Vert
	_{q(\cdot )} 	\leq C\sup_{\lambda >0}\sum_{i=1}^{m}\left\Vert \lambda \chi
	_{\{x:M_{\alpha .s}\left\vert f\right\vert ^{s}(A_{i}^{-1}x)>(\frac{\lambda}{cm})^{s}\}}\right\Vert _{q(\cdot )} 
	\end{equation*}
	So by Proposition 2.18 in \cite{C-F} and by Theorem 2.34 in \cite{C-F}
	\begin{equation*}
	=C\sup_{\lambda >0}\sum_{i=1}^{m}\left\Vert \lambda ^{s}\chi
	_{\{x:M_{\alpha .s}\left\vert f\right\vert ^{s}(A_{i}^{-1}x)>(\frac{\lambda 
		}{cm})^{s}\}}\right\Vert _{\frac{q(\cdot )}{s}}^{\frac{1}{s}} 
	\end{equation*}
	\begin{equation*}
	\leq C\sup_{\lambda >0}\sum_{i=1}^{m}\left\Vert (\frac{\lambda }{cm}%
	)^{s}\chi _{\{x:M_{\alpha .s}\left\vert f\right\vert ^{s}(A_{i}^{-1}x)>(%
\frac{\lambda }{cm})^{s}\}}\right\Vert _{\frac{q(\cdot )}{s}}^{\frac{1}{s}} 
	\end{equation*}
	\begin{equation*}
	\leq C\sup_{\lambda >0}\sum_{i=1}^{m}\left[ \sup_{\left\Vert h\right\Vert
		_{(\frac{q(\cdot )}{s})^{\prime }}=1}\int_{\{x:M_{\alpha .s}\left\vert
		f\right\vert ^{s}(A_{i}^{-1}x)>(\frac{\lambda }{cm})^{s}\}}(\frac{\lambda }{
		cm})^{s}h(x)dx\right] ^{\frac{1}{s}}
	\end{equation*}
	\begin{equation*}
	=C\sup_{\lambda >0}\sum_{i=1}^{m}\left[
	\sup_{\left\Vert h\right\Vert _{(\frac{q(\cdot )}{s})^{\prime
		}}=1} \int_{A_{i}^{-1}\{x:M_{\alpha
			.s}\left\vert f\right\vert ^{s}(A_{i}^{-1}x)>(\frac{\lambda }{cm})^{s}\}}(
	\frac{\lambda }{cm})^{s}h(A_{i}y)dy\right] ^{\frac{1}{s}}
	\end{equation*}
	\begin{equation*}
	\leq C\sup_{\lambda >0}\sum_{i=1}^{m}\left[ \sup_{\left\Vert h\right\Vert
		_{(\frac{q(\cdot )}{s})^{\prime }}=1} \int_{\{y:M_{\alpha .s}\left\vert f\right\vert ^{s}(y)>(\frac{\lambda }{cm}
		)^{s}\}}(\frac{\lambda }{cm})^{s}h(A_{i}y)dy\right] ^{\frac{1}{s}}
	\end{equation*}\newline
	We apply the Hölder's inequality, Theorem 2.26 in \cite{C-F}, 
	\begin{equation*}
	\leq C\sup_{\lambda >0}\sum_{i=1}^{m}\left[
	\sup_{\left\Vert h\right\Vert _{(\frac{q(\cdot )}{s})^{\prime
		}}=1} \left\Vert (\frac{\lambda }{cm}%
	)^{s}\chi _{\{y:M_{\alpha .s}\left\vert f\right\vert ^{s}(y)>(\frac{\lambda 
		}{cm})^{s}\}}\right\Vert _{\frac{q(\cdot )}{s}}\left\Vert h\circ
	A_{i}\right\Vert _{(\frac{q(\cdot )}{s})^{\prime }}\right] ^{\frac{1}{s}}
	\end{equation*}\newline
	Since $(\frac{q(\cdot)}{s})' \in N_{\infty}(\mathbb{R}^n)$, by Proposition \ref{prop-Ninfty}, Lemma \ref{M-F} and Proposition 2.18 in \cite{C-F}
	\begin{equation*}
	\leq C\sup_{\lambda >0}\sum_{i=1}^{m}\left[
	\sup_{\left\Vert h\right\Vert _{(\frac{q(\cdot )}{s})^{\prime
		}}=1} \left\Vert (\frac{\lambda }{cm}
	)^{s}\chi _{\{y:M_{\alpha .s}\left\vert f\right\vert ^{s}(y)>(\frac{\lambda 
		}{cm})^{s}\}}\right\Vert _{\frac{q(\cdot )}{s}}\left\Vert h\right\Vert _{(
		\frac{q(\cdot )}{s})^{\prime }}\right] ^{\frac{1}{s}}
	\end{equation*}
	\begin{equation*}
	\leq C\sup_{\lambda >0}\left\Vert (\frac{\lambda }{cm})^{s}\chi
	_{\{y:M_{\alpha .s}\left\vert f\right\vert ^{s}(y)>(\frac{\lambda }{cm}
		)^{s}\}}\right\Vert _{\frac{q(\cdot )}{s}}^{\frac{1}{s}} 	\leq C\left\Vert \left\vert f\right\vert ^{s}\right\Vert _{\frac{p(\cdot )}{s}}^{\frac{1}{s}} 
	=C\left\Vert f\right\Vert _{p(\cdot )}.
	\end{equation*}\newline\newline
$b)$ Let $f\in L_{c}^{\infty}(\mathbb{R}^n)$. By Theorem 5.54 in \cite{C-F}, since $q'(\cdot)\in N_{\infty}(\mathbb{R}^n)\cap K_{0}(\mathbb{R}^n)$, 
	
	\begin{equation*}
	\left\Vert T_{\alpha }f\right\Vert _{q(\cdot )} \leq C\left\Vert
	M^{\#}\left\vert T_{\alpha }f\right\vert \right\Vert _{q(\cdot )} 
	\end{equation*}
	Now we use (\ref{MSharp}) and since $q(\cdot)\in N_{\infty}(\mathbb{R}^n)$, by the Proposition \ref{prop-Ninfty}
	\begin{equation*}
	\leq \sum_{i=1}^{m}\left\Vert M_{\alpha ,s}f(A_{i}^{-1}\cdot )\right\Vert
	_{q(\cdot )} \leq C\sum_{i=1}^{m}\left\Vert M_{\alpha ,s}f\right\Vert _{q(\cdot )} 	=Cm\left\Vert (M_{\alpha .s}\left\vert f\right\vert ^{s})^{1/s}\right\Vert
	_{q(\cdot )} 
	\end{equation*}
	By the Proposition 2.18 in \cite{C-F}, Lemma \ref{M-F} and Proposition 2.18 in \cite{C-F}
	\begin{equation*}
	=Cm\left\Vert (M_{\alpha .s}\left\vert f\right\vert ^{s})\right\Vert _{\frac{%
			q(\cdot )}{s}}^{\frac{1}{s}} 	\leq C\left\Vert \left\vert f\right\vert ^{s}\right\Vert _{\frac{p(\cdot )}{s%
		}}^{\frac{1}{s}} =C\left\Vert f\right\Vert _{p(\cdot )}.
	\end{equation*}
	Now $b)$ follows since $L_{c}^{\infty}(\mathbb{R}^n)$ is dense in $L^{p(\cdot)}(\mathbb{R}^n)$.
\end{proof}
\begin{theo}
	Let $0\leq \alpha <n$ and let $T_{\alpha }$ be the integral operator given
	by (4). Let $m\in \mathbb{N}$ (or $m\in \mathbb{N\setminus \{}1\mathbb{\}}$
	for $\alpha =0$). Let $A_{1},...,A_{m}~$be invertible matrices such
	that $A_{i}-A_{j}$ is invertible for $i\neq j,~1\leq i,j\leq m$ and the
	functions $\Omega _{i}$ satisfy the hypothesis $(H_{1})$ and $(H_{2})$. Let $
	s\geq 1$ be defined by $\frac{1}{p_{1}}+...+\frac{1}{p_{m}}+\frac{1}{s}=1$,
	let $p(\cdot )\in \mathcal{P(\mathbb{R}}^{n}\mathcal{)}$ be such that $
	1\leq s\leq p_{-}\leq p_{+}<\frac{n}{\alpha }$ and such that $
	p(A_{i}x)=p(x)~a.e.x\in \mathbb{R}^{n}$ and let $q(\cdot )\in \mathcal{P(%
		\mathbb{R}}^{n}\mathcal{)}$ be defined by $\frac{1}{p(\cdot )}-\frac{1}{%
		q(\cdot )}=\frac{\alpha }{n}$. If the maximal operator is bounded on $%
	L^{q^{\prime }(\cdot )}(\mathbb{R}^{n})$ then,
	
	$a)$ there exist $c>0$ such that
	\begin{equation*}
	\left\Vert \lambda \chi _{\{x:T_{\alpha }f(x)>\lambda \}}\right\Vert
	_{q(\cdot )}\leq c\left\Vert f\right\Vert _{p(\cdot )} 
	\end{equation*}
	
	for all $\lambda >0$, $f\in L_{c}^{\infty }(\mathbb{R}^{n})$.
	
	$b)$ If $p_{-}>s$ then $T_{\alpha }$ extendes to a bounded operator from $
	L^{p(\cdot )}(\mathbb{R}^{n})$ into $L^{q(\cdot )}(\mathbb{R}^{n})$.
\end{theo}
\begin{proof}
	$a)$ Let $\lambda >0$ and $f\in L_{c}^{\infty }(\mathbb{R}^{n})$. By Theorem 5.54
	in $\left[ 2\right] $, since the maximal operator is bounded on $%
	L^{q^{\prime }(\cdot )}(\mathbb{R}^{n})$,%
	\begin{equation*}
	\left\Vert \lambda \chi _{\{x:T_{\alpha }f(x)>\lambda \}}\right\Vert
	_{q(\cdot )}\leq C\left\Vert \lambda \chi _{\{x:M^{\#}(T_{\alpha
		}f)(x)>\lambda \}}\right\Vert _{q(\cdot )}.
	\end{equation*}
	
	Now, by (\ref{MSharp}), as in the proof of the previous Theorem, we have that%
	\begin{equation*}
	\left\Vert \lambda \chi _{\{x:T_{\alpha }f(x)>\lambda \}}\right\Vert
	_{q(\cdot )}\leq
	\end{equation*}
	\begin{equation*}
	\leq C\sup_{\lambda >0}\sum_{i=1}^{n}\left[ \sup_{\left\Vert
		h\right\Vert _{\left( \frac{q(\cdot )}{s}\right) ^{\prime }}=1}\left\vert
	\det (A_{i}^{-1})\right\vert \left\Vert \left( \frac{\lambda }{cm}\right)
	^{s}\chi _{\{y:M_{\alpha .s}\left\vert f\right\vert ^{s}(y)>\left( \frac{%
			\lambda }{cm}\right) ^{s}\}}\right\Vert _{\frac{q(\cdot )}{s}}\left\Vert
	h\circ A_{i}\right\Vert _{\left( \frac{q(\cdot )}{s}\right) ^{\prime }}%
	\right] ^{\frac{1}{s}},
	\end{equation*}
	since $\left( \frac{q(A_{i}x)}{s}\right) ^{\prime }=\left( \frac{%
	q(x)}{s}\right) ^{\prime }$, by Proposition \ref{prop-Ninfty}, Lemma \ref{M-F} and Proposition 2.18 in \cite{C-F},
	\begin{equation*}
	\leq C\sup_{\lambda >0}\sum_{i=1}^{n}\left[ \sup_{\left\Vert h\right\Vert
		_{\left( \frac{q(\cdot )}{s}\right) ^{\prime }}=1}\left\vert \det
	(A_{i}^{-1})\right\vert \left\Vert \left( \frac{\lambda }{cm}\right)
	^{s}\chi _{\{y:M_{\alpha .s}\left\vert f\right\vert ^{s}(y)>\left( \frac{%
			\lambda }{cm}\right) ^{s}\}}\right\Vert _{\frac{q(\cdot )}{s}}\left\Vert
	h\right\Vert _{\left( \frac{q(\cdot )}{s}\right) ^{\prime }}\right] ^{\frac{1
		}{s}}
	\end{equation*}
	\begin{equation*} 
	\leq C\sup_{\lambda >0}\left\Vert \left( \frac{\lambda }{cm}\right)
	^{s}\chi _{\{y:M_{\alpha .s}\left\vert f\right\vert ^{s}(y)>\left( \frac{%
			\lambda }{cm}\right) ^{s}\}}\right\Vert _{\frac{q(\cdot )}{s}}\leq
	C\left\Vert \left\vert f\right\vert ^{s}\right\Vert _{\frac{p(\cdot )}{s}}^{%
		\frac{1}{s}}=C\left\Vert f\right\Vert _{p(\cdot )}.
	\end{equation*}
$b)$ We suppose that $s<p_{-}$. Let $f\in L_{c}^{\infty }(\mathbb{R}^{n})$. By Theorem 5.54 in \cite{C-F}, since the maximal operator is
	bounded on $L^{q^{\prime }(\cdot )}(\mathbb{R}^{n})$,
	\begin{equation*}
	\left\Vert T_{\alpha }f\right\Vert _{q(\cdot )}\leq C\left\Vert
	M^{\#}(T_{\alpha }f)\right\Vert _{q(\cdot )}
	\end{equation*}
by (\ref{MSharp}) and since $q(A_{i}x)=q(x)$, by Proposition 3,
\begin{equation*}
\leq C\sum_{i=1}^{m}\left\Vert M_{\alpha ,s}f(A_{i}^{-1}\cdot )\right\Vert
_{q(\cdot )}\leq C\sum_{i=1}^{m}\left\Vert M_{\alpha ,s}f\right\Vert
_{q(\cdot )}\leq C\left\Vert f\right\Vert _{p(\cdot )}.
\end{equation*}
where the last inequality follows as in the proof of the previous theorem.
\end{proof}

\begin{itemize}
	\item {Marta Urciuolo, FAMAF, UNIVERSIDAD NACIONAL DE CORDOBA, CIEM,
		CONICET, Ciudad Universitaria, 5000 C\'{o}rdoba, Argentina. }
	
	\textcolor{azul}{E-mail adress: urciuolo@gmail.com}
	
	\item {Lucas Vallejos, FAMAF, UNIVERSIDAD NACIONAL DE CORDOBA, CIEM,
		CONICET,  Ciudad Universitaria, 5000 C\'{o}rdoba, Argentina. }
	
	\textcolor{azul}{E-mail adress: lvallejos@famaf.unc.edu.ar}
\end{itemize}


\end{document}